\documentclass[11pt]{article}
\usepackage{ amssymb }
\usepackage{lipsum}

\usepackage{amsfonts}
\usepackage{amsthm}

\usepackage{epsfig}
\usepackage{graphicx}  
  
\usepackage{pdfsync}

\usepackage{caption}
\usepackage{subcaption}
 
\usepackage{amsmath,amsfonts,amssymb,amsthm,epsfig,epstopdf,titling,url,array}

 \usepackage[utf8]{inputenc}

\usepackage{hyperref}

\theoremstyle{plain}
\newtheorem{thm}{Theorem}[section]

\newtheorem{prop}[thm]{Proposition}

\theoremstyle{definition}

\newtheorem{conj}{Conjecture}[section]

\theoremstyle{remark}

\theoremstyle{remark}

\usepackage{graphicx}
\usepackage{xcolor}
\definecolor{darkblue}{rgb}{0,0,0.5} 
\usepackage{transparent}

\begin{document}

\title{Asymptotic dimension of Artin groups and a new upper bound for Coxeter groups.}
%\title{Asymptotic dimension of Artin and Coxeter groups.}
%\author{Panagiotis Tselekidis \\ panagiotis.tselekidis@maths.ox.ac.uk}

\author{Panagiotis Tselekidis}

%\author{Panagiotis Tselekidis \\ email: \href{mailto:panagiotis.tselekidis@maths.ox.ac.uk}{panagiotis.tselekidis@maths.ox.ac.uk}} 

\maketitle

%\href{mailto:panagiotis.tselekidis@maths.ox.ac.uk}{panagiotis.tselekidis@maths.ox.ac.uk}
%\email{\textit{email:} panagiotis.tselekidis@maths.ox.ac.uk}

%\address{Mathematical Institute, University of Oxford, 24-29 St Giles',
%Oxford, OX1 3LB, U.K.}

\begin{abstract}
If $A_\Gamma$ ($W_\Gamma$) is the Artin (Coxeter) group with defining graph $\Gamma$ we denote by $Sim(\Gamma)$ the number of vertices of the largest clique in $\Gamma$. We show that $asdimA_\Gamma \leq Sim(\Gamma)$, if $Sim(\Gamma)=2$. We conjecture that the inequality holds for every Artin group. We prove that if for all free of infinity Artin (Coxeter) groups the conjecture holds, then it holds for all Artin (Coxeter) groups. 
As a corollary, we show that $asdimW_\Gamma \leq Sim(\Gamma)$ for all Coxeter groups, which is the best known upper bound for the asymptotic dimension of Coxeter Groups. As a further corollary, we show that the asymptotic dimension of any Artin group of large type with $Sim(\Gamma)=3$ is exactly two.

%If $A_\Gamma$ ($W_\Gamma$) is the Artin (Coxeter) group with underline graph $\Gamma$ we denote by $Sim(\Gamma)$ the number of vertices of the largest clique in $\Gamma$. We show that $asdimA_\Gamma \leq Sim(\Gamma)$ if $Sim(\Gamma)=2$. We conjecture that the inequality holds for every Artin group. We prove that if for all free of infinity Artin (Coxeter) groups the conjecture holds, then it holds for all Artin (Coxeter) groups. 
%As a corollary we show that $asdimW_\Gamma \leq Sim(\Gamma)$ for all Coxeter groups, which is the best known upper bound for the asymptotic dimension of Coxeter Groups.

\end{abstract}

\tableofcontents

\section{Introduction}
%\section{Asymptotic dimension of Artin groups and a new upper bound for Coxeter groups.}

Artin groups are easily defined via presentations, like Coxeter groups, but they are often very poorly understood. There are numerous conjectures about them, for example, it is not known in general whether they are torsion free. While Artin groups remain much more mysterious
than their Coxeter relatives, a driving theme behind current research has been to show that
these groups are as well behaved as Coxeter groups.

Some examples of Artin groups are the free non-abelian groups, free abelian groups, and RAAGs. 
In recent years, Artin and Coxeter groups proved to be a very rich source of examples and counterexamples of interesting
phenomena in geometric group theory.

%While many problems about these groups are still wide open in general, our understanding of the structure and the geometry of Artin groups has seen an explosion after 2010.

In this paper we are interested in the geometry of Artin and Coxeter groups and in particular in the asymptotic dimension of these groups. The asymptotic dimension achieved particular prominence in geometric group theory after a paper of Guoliang Yu, (see \cite{Yu}) which proved the Novikov higher signature conjecture for manifolds whose fundamental group has finite asymptotic dimension.
%Unfortunately, not all finitely presented groups have finite asymptotic dimension. For example, Thompson's group $F$ has infinite
%asymptotic dimension since it contains $\mathbb{Z}^{n}$ for all $n$.\\
It is known for many classes of groups that they have finite asymptotic dimension, for instance, hyperbolic, relative hyperbolic, hierarchically hyperbolic, mapping class groups of surfaces, and one relator groups have finite asymptotic dimension (see \cite{BD08}, \cite{Os}, \cite{HHS}, \cite{BBF}, \cite{Mats}, \cite{PT}).

However, there are still many open questions regarding asymptotic dimension, for example,\\
Do all Artin groups have finite asymptotic dimension? \cite{Dra}\\
%2. do all $CAT(0)$ groups have finite asymptotic dimension? \cite{Dra}\\
Do all $Out(F_n )$ groups have finite asymptotic dimension, when $n>2$? \cite{Dra}\\
%4. do all automatic groups have finite asymptotic dimension? \cite{Dra}\\

Recently, it was proved in \cite{HS} that Artin groups being both of large and of hyperbolic type are hierarchically hyperbolic. As a consequence, these Artin groups have finite asymptotic dimension. However, no specific bound was obtained.\\

The purpose of this paper is to give some new results about the asymptotic dimension of Artin and Coxeter groups.
To be more precise, let $\Gamma$ be a finite simplicial labeled graph, we denote by $A_\Gamma$ ($W_\Gamma$) the \textit{Artin group} (\textit{Coxeter group}) associated to the graph $\Gamma$.
%We define $Val(\Gamma)= max \lbrace  valency(u) \mid u \in V(\Gamma) \rbrace$. By $valency(u)$ of a vertex we denote the number of edges incident to the vertex $u$. Then applying the Theorem 1.1 we obtain that: $$asdimA(\Gamma) \leq Val(\Gamma) +1.$$
We set 
\begin{center}
$Sim(\Gamma)= max \lbrace n \mid $ $\Gamma$ contains the 1-skeleton of the standard $(n-1)$-simplex $\Delta^{n-1} \rbrace.$
\end{center} 
We conjecture that:
\begin{conj}\label{Artin.Q1.1}
$asdimA_\Gamma \leq Sim(\Gamma)$, for all Artin groups.
\end{conj}

\begin{conj}\label{Artin.Q1.2}
$asdimW_\Gamma \leq Sim(\Gamma)$, for all Coxeter groups.
\end{conj}

%If the conjecrure is false, then we can ask

%\begin{flushleft}
%\textbf{Question 1.1}\label{QU1.1} For what Artin groups the conjecture is true?\\
%\end{flushleft}

It is known that if $A_\Gamma$ is a RAAG, then $asdimA_\Gamma=Sim(\Gamma)$ see (\cite{PT} and \cite{Wr}). Since when $Sim(\Gamma)=1$ the group $A_\Gamma$ is a free group, $asdim(A_\Gamma)=1$. We prove that:

\begin{thm}\label{Artin.T1.1}
Let $A_\Gamma $ be an Artin group with $Sim(\Gamma)=2$. Then 
\begin{center}
  $asdim A_\Gamma  =2$.
\end{center}
\end{thm}

\begin{thm}\label{Artin.T1.2}
Let $A_\Gamma $ be an Artin group of large type with $Sim(\Gamma)=3$. Then
 \begin{center}
 $ asdim A_\Gamma  =2$.
\end{center}
\end{thm}

It is known by an isometric embedding theorem of Januszkiewicz (see \cite{TJ}) that Coxeter groups have finite asymptotic dimension. In particular, Januszkiewicz's theorem shows that for any Coxeter group $W_\Gamma$ with defining graph $\Gamma$ we have the following upper bound: $asdimW_\Gamma \leq \sharp V(\Gamma)$.\\
Dranishnikov notes that this can be improved to the following: $asdimW_\Gamma \leq ch(N(W_\Gamma ))$.
Where $N(W_\Gamma)$ is the nerve
of $W_\Gamma$ and $ch(N(W_{\Gamma} ))$ is the chromatic number of the 1-skeleton of $N(W_{\Gamma} )$, we note that the 1-skeleton of $N(W_{\Gamma} )$ is $\Gamma$. In general, $Sim(\Gamma) \leq ch( \Gamma ) $.\\

%Since dim N + 1 ≤ ch(N) for every simplicial complex N

Using theorem \ref{Artin.T1.1} we prove that:
\begin{thm}\label{Artin.T1.3}
If for all free of infinity Artin (Coxeter) groups the conjecture \ref{Artin.Q1.1} (conjecture \ref{Artin.Q1.2}) holds, then it holds for all Artin (Coxeter) groups.
\end{thm}

We recall that an Artin (Coxeter) group is \textit{free of infinity} if the corresponding graph is complete. Using theorem \ref{Artin.T1.3} we show that conjecture \ref{Artin.Q1.2} is true.  

\begin{thm}\label{Artin.T1.4}
Let $W_\Gamma $ be a Coxeter group. Then
 \begin{center}
 $asdim W_\Gamma  \leq Sim(\Gamma)$.
\end{center}
\end{thm}

Since $Sim(\Gamma) \leq ch( N(W_{\Gamma}  )$ and there are graphs such that the inequality is strict, theorem \ref{Artin.T1.4} improves the previously best known upper bound.\\

By a finite simplicial \emph{G-labeled graph} we mean a finite simplicial graph $\Gamma$ such that every
edge $[a,b]$ is labeled by a cyclically reduced word $r_{a,b}$ (in terms of letters-verices $a,b$). The word  $r_{a,b}$ is neither of the form $b^k a b^\lambda$ nor $a^k b a^\lambda$.

The \emph{Graph group} associated to a finite simplicial G-labeled graph $\Gamma$ is the group $G_\Gamma$ given by the following presentation:
\begin{center}
$G_\Gamma= \langle a,b  \vert r_{a,b}  $ when $a,b$ are connected by an edge $ \rangle $.
\end{center}
Obviously, Graph groups generalize Artin groups.  
  
\begin{prop}\label{Artin.P3.2.} Let $\Gamma$ be a finite simplicial G-graph such that $Sim(\Gamma)=2$. We further assume that $H_1(\Gamma, \mathbb{Z} )=0$. Then 
\begin{center}
$asdim\,G_\Gamma \leq 2 .$
\end{center}
\end{prop}

\section{Preliminaries.}
The asymptotic dimension $asdimX$ of a metric space $X$ is defined as follows: $asdimX \leq n$ if and only if for every $R > 0$ there exists a uniformly bounded covering $\mathcal{U}$ of $X$ such that the R-multiplicity of $\mathcal{U}$ is smaller than or equal to $n+1$ (i.e. every R-ball in $X$ intersects at most $n+1$ elements of $\mathcal{U}$).

There are many equivalent ways to define the asymptotic dimension of a metric space. It turns out that the asymptotic dimension of an infinite tree is $1$ and the asymptotic dimension of $\mathbb{E}^{n}$ is $n$.
\begin{flushleft}
By a finite simplicial labeled graph we mean a finite simplicial graph $\Gamma$ such that every 
%vertex is labeled by an element of $S$ such that $\sharp S= \sharp V(\Gamma)$ and every
edge $[a,b]$ is labeled by a natural number $m_{ab}>1$. The Artin group
associated to $\Gamma$ is the group $A_\Gamma$ given by the following presentation:
\end{flushleft}

\begin{center}
$A_\Gamma= \langle a,b \in V(\Gamma) \vert  \underbrace{ab \ldots}_{m_{ab}}= \underbrace{ba \ldots}_{m_{ab}} $ when a, b are connected by an edge $ \rangle $.

\end{center}
\textbf{Convention.} We will use the same symbol $m_{ab}$ both for the label of the edge $[a,b]$ and the relation  $\underbrace{ab \ldots}_{m_{ab}} \cdot \underbrace{a^{-1}b^{-1} \ldots}_{m_{ab}}$ of the group.\\

We call an Artin group\textit{ Right-angled Artin group} (RAAG) if $m_{ab}=2$ when a, b are connected by an edge. We call an Artin group \textit{free of infinity} if the corresponding graph is complete. An Artin group is of \emph{large type} if every label $m_{ab}$ is at least three.

\begin{flushleft}
The Coxeter group
associated to $\Gamma$ is the group $W_\Gamma$ given by the following presentation:
\end{flushleft}

\begin{center}
$W_\Gamma= \langle V(\Gamma) \vert a^2 $ for all $a \in V(\Gamma)$ and $  \underbrace{ab \ldots}_{m_{ab}}= \underbrace{ba \ldots}_{m_{ab}} $ when a, b are connected by an edge $ \rangle $.

\end{center}
%\textbf{Convention.} We will use the same symbol $m_{ab}$ both for the label of the edge $[a,b]$ and the relation  $\underbrace{ab \ldots}_{m_{ab}} \cdot \underbrace{a^{-1}b^{-1} \ldots}_{m_{ab}}$ of the group.\\
A Coxeter group is called\textit{ Right-angled Coxeter group} (RACG) if $m_{ab}=2$ when a, b are connected by an edge. We call a Coxeter group \textit{free of infinity} if the corresponding graph is complete.\\

%We call an Artin group \textit{free of infinity} if the corresponding graph is complete. We call an Artin group \emph{even} if every label $m_{ab}$ is even, \emph{odd} if every label $m_{ab}$ is odd, and \emph{constant} if there exists $M>0$ such that $m_{ab}=M$ for every edge.\\
We say a simplicial graph is \emph{complete} or equivalently \emph{a clique} if any two vetrices are connected by an edge. An \emph{n-clique} is the complete graph on n vertices. There is another equivalent way to define $Sim(\Gamma)$, in particular, it can be viewed as the larger natural number $n$ such that $\Gamma$ contains a complete subgraph $\Gamma^{\prime}$ with $n$ vertices.\\
We recall that the \emph{full subgraph} defined by a subset $V$ of the vertices of a graph $\Gamma$ is a subgraph of $\Gamma$ formed from $V$ and from all of the edges of $\Gamma$ which have both endpoints in the subset $V$.
%Let $G$ be a subgraph of $\Gamma$, we denote by $FS^{\Gamma}(G)$ the full-subgraph of $\Gamma$ defined by $V(G)$.\\
%The \emph{simplicial closure} of $\Gamma$ is the flag complex $SC(\Gamma)$ defined by $\Gamma$.\\

The following theorem is proved by the author in $\cite{PT}$
\begin{thm}\label{Artin.T2.1}
Let $(\mathbb{G}, Y)$ be a finite graph of groups with vertex groups $\lbrace G_{v} \mid v \in Y^{0} \rbrace$ and edge groups $\lbrace G_{e} \mid e \in Y^{1}_{+} \rbrace$. Then the following inequality holds:
\begin{center}
$asdim\,\pi_{1}(\mathbb{G},Y,\mathbb{T})  \leq max_{v \in Y^{0} ,e \in Y^{1}_{+}} \lbrace asdim G_{v}, asdim\,G_{e} +1 \rbrace.$
\end{center}

\end{thm}

 We also need a theorem for free products from \cite{BDK}, see also \cite{BD04}.
\begin{thm}\label{Artin.T2.2}
Let $A$, $B$ be two finitely generated groups. Then:
\begin{center}
$asdim\,A \ast B  = max \lbrace asdim A, asdim B, 1 \rbrace.$
\end{center}

\end{thm}

We recall also the following (see \cite{PT}) 

\begin{thm}\label{Artin.T2.3} 
 \textit{Let $G$ be a finitely generated one relator group. Then one of the following is true}:\\
\textbf{(i)} \textit{$G$ is finite cyclic, and $asdim\,G = 0$} \\
\textbf{(ii)} \textit{$G$ is a nontrivial free group or a free product of a nontrivial free group and a finite cyclic group, and $asdim\,G = 1$}\\
\textbf{(iii)} \textit{$G$ is an infinite freely indecomposable not cyclic group or a free product of a nontrivial free group and an infinite freely indecomposable not cyclic group, and $asdim\,G = 2$.}
\end{thm}

\section{The case $Sim(\Gamma)=2.$}
Let $\Gamma$ be a finite simplicial labeled graph, such that $Sim(\Gamma)=2$. We distinguish two cases, either the graph contains a cycle or not. We will handle these cases separately.

\begin{prop}\label{Artin.P3.1} Let $\Gamma$ be a finite simplicial labeled graph, such that $Sim(\Gamma)=2$. We further assume that $H_1(\Gamma, \mathbb{Z} )=0$, then 
\begin{center}
$asdim\,A_\Gamma=2 $ and $asdim\,W_\Gamma \leq 1 $.
\end{center}

\end{prop} 
\begin{proof}
(\emph{Artin groups}) We observe that if $\Gamma$ is not connected, then  $A_\Gamma = \ast_{i=1}^{n} A_{\Gamma_i } $, where $\Gamma_i$ are the connected components of $\Gamma$. By the previous observation and theorem \ref{Artin.T2.2} we may consider only connected graphs. So we assume that $\Gamma$ is connected, thus $\Gamma$ is a tree.

We observe that if $\Gamma$ is a tree, then $A_\Gamma$ is the fundamental group of graph of groups where the vertex groups are one-relator parabolic subgroups of $A_\Gamma$ and the edge groups are infinite cyclic parabolic subgroups. 
By theorem \ref{Artin.T2.1} and theorem \ref{Artin.T2.3} we have that $asdim\,A_\Gamma \leq 2.$.

 Suppose that $A_\Gamma$ is the artin group defined by a single edge. Obviously, $A_\Gamma$ is one relator group. Since $A_\Gamma$ is not a free group and it is torsion free (see \cite{LynSch} page 200), by theorem \ref{Artin.T2.3} we obtain that $asdim\,A_\Gamma =2$. Since every Artin group with $Sim(\Gamma)=2$ contains an one relator parabolic subgroup we conclude that the proposition is true.

(\emph{Coxeter groups}) If $\Gamma$ is a tree, then $W_\Gamma$ is the fundamental group of graph of finite parabolic subgroups. 
By theorem \ref{Artin.T2.1} we have that $asdim\,W_\Gamma \leq 1$.

\end{proof}

In the second case we have that the graph contains a cycle, however, it is possible to contain only cycles of length larger than or equal to four.

\begin{thm}\label{Artin.T3.2} Let $\Gamma$ be a finite simplicial labeled graph, such that $Sim(\Gamma)=2$. Then 
 
\begin{center}
$asdim\,A_\Gamma=2 $ and $asdim\,W_\Gamma \leq 2 $.
\end{center}

\end{thm} 
\begin{proof} Again we may assume that $\Gamma$ is connected. We will use induction on the number $n_C$ of cycles contained in the graph. If $n_C=0$, then by proposition \ref{Artin.P3.1} the statement of this theorem is true. We assume that the statement is true for all $n_C \leq n$. Let $\Gamma$ be a finite simplicial labeled graph such that $n_C =n+1$. By the previous observation we know that the graph contains a cycle $\gamma$ of length $m+1\geq4$.

Let  $\lbrace v_1, \ldots, v_{m+1} \rbrace$ be the vertex set of $\gamma$ and $\lbrace  e_1, \ldots, e_{m+1} \rbrace$ be the edges of the cycle where $e_i = [v_i,v_{i+1}]$. We set $\Gamma_1 = \Gamma \setminus \lbrace  v_1 \rbrace $, obviously this graph contains at least one cycle less than $\Gamma$, thus $asdim\, A_{\Gamma_1 }=2$ ($asdim\, W_{\Gamma_1 } \leq 2$).

Let $U= \lbrace  u_1, \ldots, u_{k}  \rbrace$ be all the vertices of $\Gamma$ such that for each $u_j$ there exists an edge connecting it with $v_1$. It is easy to observe that the full subgraph $\Gamma_U$ of $\Gamma$ formed by $U$  contains no edge. Thus the parabolic subgroup $A_{\Gamma_U }$ ($W_{\Gamma_U }$) of $A_\Gamma$ ($W_\Gamma$) is the free group with $k$ generators $F_k$ (is the virtually free group $\ast_{i=1}^{k} \mathbb{Z}_2$). 

Let $\Gamma_{v_1 }$ be the full subgraph of $\Gamma$ formed by $U \cup \lbrace v_1  \rbrace$, we observe that contains no cycle. It is not hard to see that $\Gamma_{v_1 }$ contains at least an edge, thus $asdim\, A_{\Gamma_{v_1} }=2$ ($asdim\, W_{\Gamma_{v_1} }\leq 1$.).

It is easy to check that the Artin (Coxeter) group $A_\Gamma$ ($W_\Gamma$) is an amalgamated product of its parabolic subgroups  $A_{\Gamma_1 }$ ($W_{\Gamma_1 }$), $A_{\Gamma_U }$ ( $W_{\Gamma_U }$ ) and $ A_{\Gamma_{v_1 }}$ ($ W_{\Gamma_{v_1 }}$ ). To be more precise, the following equalities are a direct consequence of the presentation of $A_\Gamma$ and $W_\Gamma$:

\begin{center}
$A_\Gamma = A_{\Gamma_1 } \underset{A_{\Gamma_U }}{\ast} A_{\Gamma_{v_1 } } $ and $W_\Gamma = W_{\Gamma_1 } \underset{W_{\Gamma_U }}{\ast} W_{\Gamma_{v_1 } }$
\end{center}

Then by theorem \ref{Artin.T2.1} (or using Dranishnikov's theorem (see \cite{Dra08})) we conclude that the theorem holds.
\end{proof}

\subsection{Graph-Groups.}
In fact we may obtain a more general result than proposition \ref{Artin.P3.1}.

By a finite simplicial \emph{G-labeled graph} we mean a finite simplicial graph $\Gamma$ such that every
edge $[a,b]$ is labeled by a cyclically reduced word $r_{a,b}$ (in terms of letters-verices $a,b$). The word  $r_{a,b}$ is neither of the form $b^k a b^\lambda$ nor $a^k b a^\lambda$.

The \emph{Graph group} associated to a finite simplicial G-labeled graph $\Gamma$ is the group $G_\Gamma$ given by the following presentation:
\begin{center}
$G_\Gamma= \langle a,b  \vert r_{a,b}  $ when $a,b$ are connected by an edge $ \rangle $.
\end{center}
Obviously, Artin groups are Graph groups.

\begin{prop}\label{Artin.P3.2} Let $\Gamma$ be a finite simplicial G-graph, such that $Sim(\Gamma)=2$. We further assume that $H_1(\Gamma, \mathbb{Z} )=0$, then 
\begin{center}
$asdim\,G_\Gamma \leq 2 .$
\end{center}

\end{prop} 
\begin{proof}
 We assume that $\Gamma$ is connected, thus $\Gamma$ is a tree. Suppose that $G_\Gamma$ is a graph group defined by a single edge. Obviously, $G_\Gamma$ is one relator group. By theorem \ref{Artin.T2.3} we obtain that $asdim\,G_\Gamma \leq 2$.
 
It is important to note that if $a$ is a vertex of $\Gamma$, then $\langle a \rangle$ is an infinite cyclic subgroup of $G_\Gamma$. This follows by the Freiheitssatz (see [23], thm 5.1, page 198) and induction on the number of vertices. Since $Sim(\Gamma)=2$, the graph contains at least an edge.

We observe that if $\Gamma$ is a tree, then $G_\Gamma$ is the fundamental group of graph of groups where the vertex groups are one-relator Graph subgroups of $G_\Gamma$ and the edge groups are infinite cyclic subgroups. 
By theorem \ref{Artin.T2.1} and theorem \ref{Artin.T2.3} we have that $asdim\,G_\Gamma \leq 2.$.

\end{proof}

\section{The main results.}

%We observe that the free of infinity Artin groups are those defined by a complete graph.

\begin{thm}\label{Artin.T5.1}
If for all free of infinity Artin (Coxeter) groups the conjecture \ref{Artin.Q1.1} (conjecture \ref{Artin.Q1.2}) holds, then it holds for all Artin (Coxeter) groups.
\end{thm}
\begin{proof} We will show the theorem only for Artin groups since the proof for Coxeter groups is similar to that for Artin groups. 

We observe that the free of infinity Artin groups are those defined by a complete graph. 
By theorem \ref{Artin.T2.2} it suffices to prove the statement only for Artin groups defined by a connected graph.

Let $\Gamma$ be a finite simplicial connected labeled graph. We will use induction on $Sim(\Gamma)$. We distinguish two cases.\\
\textbf{Case 1.} The graph has at most two vertices.\\
Then $ A_\Gamma $ is either a free group or an one relator group. In the first subcase the asymptotic dimension is one and in the second subcase by theorem \ref{Artin.T3.2} the asymptotic dimension is two. In particular in both subcases the conjecture \ref{Artin.Q1.1} holds.\\
\textbf{Case 2.} The graph has at least three vertices.\\
For $Sim(\Gamma)=1 $ or $2$ the conjecture \ref{Artin.Q1.1} true, we assume that the statement of this theorem is true for any $Sim(\Gamma) \leq n$, where $n \geq 2 $. Let $\Gamma$ be a finite simplicial connected labeled graph with $Sim(\Gamma)=n+1$.

It is vital to use another induction here, in particular, we use induction on $n_{V(\Gamma)} = \sharp V(\Gamma)$.
If $n_{V(\Gamma)} = Sim(\Gamma)$, then the graph is a complete graph thus the group $A_\Gamma$ is free of infinity, so $asdim A_\Gamma  \leq n+1$. We assume that $asdim A_\Gamma  \leq n+1 $ for every graph with $n_{V(\Gamma)}< N+1$, we note that $N \geq Sim(\Gamma)$. Let $\Gamma$ be a finite simplicial connected labeled graph with $Sim(\Gamma)=n+1$ and $n_{V(\Gamma)}=N+1$. 

Let $\lbrace C_1 ,\ldots ,C_k \rbrace$ be the complete subgraphs of $\Gamma$ such that $Sim(C_i)=Sim(\Gamma)$, we will call these subgraphs \emph{maximal complete subgraphs}. We set $n_C = \sharp \lbrace C_1 ,\ldots ,C_k \rbrace$. We further distinguish two subcases.

\textbf{Subcase 2.1} $n_C =1$.\\
Let $C$ be the only maximal complete subgraph of $\Gamma$. Since $n_{V(\Gamma)} > Sim(\Gamma)$, there exist vertices $v_1, v_2 \in C$ and $x \notin C $ such that there is no edge connecting $v_1$ with $x$. We set $\Gamma_1 = \Gamma \setminus \lbrace  v_1  \rbrace $, $\Gamma_x = \Gamma \setminus \lbrace  x  \rbrace $ and $\Gamma_{1x} = \Gamma \setminus \lbrace  v_1 , x \rbrace $ and we observe that all these graphs form the follwing parabolic subgroups of $A_\Gamma$, $A_{\Gamma_1 }$ $A_{\Gamma_x }$ and $A_{\Gamma_{1x} }$. As a direct consequence of the presentation of $A_\Gamma$ we have that 

\begin{equation}\label{Artin.E1}
A_\Gamma = A_{\Gamma_1 } \underset{A_{\Gamma_{1x} }}{\ast} A_{\Gamma_x }.
\end{equation}
It is easy to see that $Sim(\Gamma_1), Sim(\Gamma_{1x} )< n+1 $, while $Sim( \Gamma_{1x} )=n+1$ but $n_{V(\Gamma_{1x})} < n_{V(\Gamma)}=N+1 $. Thus  $asdim\,A_{\Gamma_1 } , asdim\,A_{\Gamma_{1x} }< n+1$ while $asdim\,A_{\Gamma_x }  \leq n+1$.
Applying theorem \ref{Artin.T2.1} we conclude that $asdim\,A_\Gamma \leq n+1.$

\textbf{Subcase 2.2} $n_C >1$.\\
This case is a bit more complicate than the previous, the goal is to write the $A_\Gamma$ as an amalgamatated product, similar to the equality (\ref{Artin.E1}).

Recall that $S_C = \lbrace C_1 , \ldots ,C_{n_{C}} \rbrace$ is the set of the maximal complete subgraphs of $\Gamma$. Observe that there exists a vertex $v_2$ of $C_2$ and an edge $e=[v_0 , v_1]$ of $C_1$ such that there is no edge between $v_2$ and $v_1$. We will divide $S_C$ into two disjoint subsets $S_1$ and $S_2$, the first set consists of the maximal subgraphs containing $v_1$ and the second set containing the rest of the elements. After possibly relabeling the set $S_C$ we may assume that $S_2=  \lbrace C_2 , \ldots , C_{\sigma} \rbrace $. We fix an element $v_i$ for every $C_i \in S_2$ such that there is no edge between $v_1$ and $v_i$ (such elements exist and they are not equal to $v_0$).

We set $\Gamma_1 = \Gamma \setminus \lbrace  v_1  \rbrace $, $\Gamma_\sigma = \Gamma \setminus \lbrace   v_2 , \ldots , v_{\sigma}  \rbrace $ and $\Gamma_{1 \sigma} = \Gamma \setminus \lbrace  v_1 , \ldots ,v_\sigma \rbrace $. We observe that all these graphs form the follwing parabolic subgroups of $A_\Gamma$, $A_{\Gamma_1 }$ $A_{\Gamma_\sigma }$ and $A_{\Gamma_{1 \sigma} }$. As a direct consequence of the presentation of $A_\Gamma$ we have that 

\begin{equation}\label{Artin.E2}
A_\Gamma = A_{\Gamma_1 } \underset{A_{\Gamma_{1 \sigma} }}{\ast} A_{\Gamma_\sigma }.
\end{equation}
Observe that $Sim(\Gamma_{1 \sigma} )< n+1 $, while $Sim( \Gamma_{1} )= Sim(\Gamma_\sigma )=n+1$, however, $n_{V(\Gamma_{\sigma})}, n_{V(\Gamma_{1})} < n_{V(\Gamma)}=N+1 $. So we obtain $ asdim\,A_{\Gamma_{1\sigma} }< n+1$ while $ asdim\,A_{\Gamma_\sigma } , asdim\,A_{\Gamma_1 }  \leq n+1$.
By theorem \ref{Artin.T2.1} we conclude that $asdim\,A_\Gamma \leq n+1.$

This completes the proof that $asdim\,A_\Gamma \leq n+1$ in any subcase when $Sim(\Gamma)=n+1$, thus by induction the theorem holds.

\end{proof}

\begin{thm}\label{Artin1.11}
Let $\Gamma$ be a finite simplicial labeled graph, and let $W_\Gamma $ be the Coxeter group defined by $\Gamma$. Then
 \begin{center}
 $asdim W_\Gamma  \leq Sim(\Gamma)$.
\end{center}
\end{thm}
\begin{proof} By an isometric embedding theorem of Januszkiewicz (see \cite{TJ}) we have that $asdimW_\Gamma \leq \sharp V(\Gamma)$. Observe that if $\Gamma$ is complete graph, then $\sharp V(\Gamma)=Sim(\Gamma)$. This means that the conjecture \ref{Artin.Q1.2} is true for complete graphs.\\
Thus by theorem \ref{Artin.T5.1} we conclude that conjecture \ref{Artin.Q1.2} is true for any Coxeter group.

\end{proof}

As a corollary of theorem \ref{Artin.T5.1} we have:

\begin{prop}\label{Artin1.12}
Let $A_\Gamma $ be an Artin group of large type with $Sim(\Gamma)=3$. Then
 \begin{center}
 $ asdim A_\Gamma =2$.
\end{center}
\end{prop}
\begin{proof} By \cite{KJ} every Artin group of large type defined by a triangle splits as a graph of free groups, thus by theorem \ref{Artin.T2.1} $ asdim A_\Gamma =2$. By the proof of theorem \ref{Artin.T5.1} and the fact that every parabolic subgroup of a large type Artin group is also a large type Artin group, we conclude that the proposition is true.

\end{proof}

\textit{E-mail}: panagiotis.tselekidis@queens.ox.ac.uk

%\textit{E-mail address 2}: Panagiotis.Tselekidis@maths.ox.ac.uk\\

\textit{Address:} Mathematical Institute, University of Oxford, Andrew Wiles Building, Woodstock Rd, Oxford OX2 6GG, U.K.

%\textit{Address 2:} The Queen's College, High Street,
%Oxford, OX1 4AW, U.K.

\end{document}